\documentclass[12pt]{article}
\usepackage[mathscr]{eucal}
\usepackage{amssymb}
\usepackage{latexsym}
\usepackage{amsthm}
\usepackage{amsmath}
\usepackage[dvipdfmx]{graphicx}
\usepackage{psfrag}
\usepackage{a4wide}

\newcommand{\R}{\mathbb{R}}

\newcommand{\la}{\langle}
\newcommand{\ra}{\rangle}

\newtheorem{theorem}{Theorem}[section]
\newtheorem{proposition}[theorem]{Proposition}
\newtheorem{lemma}[theorem]{Lemma}
\newtheorem{corollary}[theorem]{Corollary}
\theoremstyle{definition}
\newtheorem{definition}[theorem]{Definition}

\newtheorem{remark}[theorem]{Remark}

\newtheorem{example}[theorem]{Example}

\makeatletter
\@addtoreset{equation}{section}

\makeatother

\title{Algebro-combinatorial generalizations of the Victoir method for constructing weighted designs}
\author{Hiroshi Nozaki\thanks{Department of Mathematics Education, 
	Aichi University of Education, 
	1 Hirosawa, Igaya-cho, 
	Kariya, Aichi 448-8542, 
	Japan. {\tt hnozaki@auecc.aichi-edu.ac.jp}} and 
    Masanori Sawa\thanks{
    Graduate School of System Informatics, Kobe University, 1-1 Rokkodai, Nada, Kobe 657-
8501, Japan. 
    {\tt sawa@people.kobe-u.ac.jp}}}

\begin{document}

\maketitle

\renewcommand{\thefootnote}{\fnsymbol{footnote}}
\footnote[0]{2020 Mathematics Subject Classification: 65D32 (05B15, 05E30)
}

\begin{abstract}
  A weighted $t$-design in $\mathbb{R}^d$ is a finite weighted set that exactly integrates all polynomials of degree at most $t$ with respect to
  a given probability measure.
   A fundamental problem is to construct weighted $t$-designs with as few points as possible. 
  Victoir (2004) proposed a method to reduce the size of weighted $t$-designs while preserving the $t$-design property by using combinatorial objects such as combinatorial designs or orthogonal arrays with two levels.
  In this paper, we give an algebro-combinatorial generalization of both Victoir's method and its variant by the present authors (2014) in the framework of Euclidean polynomial spaces, enabling us to reduce the size of weighted designs obtained from the classical product rule. 
  Our generalization allows the use of orthogonal arrays with arbitrary levels, whereas Victoir only treated the case of two levels.
  As an application, we present a construction of equi-weighted $5$-designs with $O(d^4)$ points for product measures such as Gaussian measure $\pi^{-d/2} e^{-\sum_{i=1}^d x_i^2} dx_1 \cdots dx_d$ on $\mathbb{R}^d$ or equilibrium measure $\pi^{-d} \prod_{i=1}^d (1-x_i^2)^{-1/2} dx_1 \cdots dx_d$ on $(-1,1)^d$, where $d$ is any integer at least 5. 
  The construction is explicit and does not rely on numerical approximations.
  Moreover, we establish an existence theorem of Gaussian $t$-designs with $N$ points for any $t \geq 2$, where $N< q^{t}d^{t-1}=O(d^{t-1})$ for fixed sufficiently large prime power $q$. As a corollary of this result, we give an improvement of a famous theorem by Milman (1988) on isometric embeddings of the classical finite-dimensional Banach spaces.
 \end{abstract}
\bigskip

\noindent
\textbf{Keywords:}
Weighted design, Euclidean polynomial space, Product rule, Orthogonal array, Linear code, Gaussian design, Equilibrium design, Isometric embedding, Spherical design, Association scheme

\section{Introduction} \label{sec:introduction}

A {\it weighted $t$-design} in $\R^d$ with respect to
a probability measure $\mu$
is a finite subset $X \subset \R^d$ with a positive weight function $w : X \to \mathbb{R}_{>0}$ such that
\[
\int_{\R^d} f(x) \, d\mu(x) = \sum_{x \in X} w(x) f(x)
\]
for all polynomials $f \in \mathbb{R}_t[x_1, \ldots, x_d]$,
where $\mathbb{R}_t[x_1, \ldots, x_d]$ denotes the space of real polynomials of total degree at most $t$ in $d$ variables; in particular, this is called an {\it (equi-weighted) $t$-design} if $w$ is constant.
We refer the reader to Bannai et al.~\cite{BBIT2021} for a good introduction to the basics on design theory in algebraic combinatorics.
On the other hand, a weighted design has long been studied in the context of cubature formula in numerical analysis and related areas, with emphasis on the connection to the theory of orthogonal polynomials.
A comprehensive textbook for such analytic aspects of design theory is Dunkl and Xu~\cite{DX2014}.

One of the major problems in design theory is to construct a small-sized weighted $t$-design for given values of $d$ and $t$.  
Suppose a measure $\mu$ is expressed in terms of a density function $\phi(x)$, that is, $d\mu(x) = \phi(x)\,dx$.  
The product rule constructs weighted $t$-designs as Cartesian products of lower-dimensional weighted $t$-designs,  
provided that the density function $\phi(x)$ can be written as a Cartesian product of lower-dimensional density functions.
This is exemplified by the Gaussian measure $\pi^{-d/2} e^{-\sum_{i=1}^d x_i^2} dx_1 \cdots dx_d$ on $\mathbb{R}^d$ and the equilibrium measure $\pi^{-d} \prod_{i=1}^d (1-x_i^2)^{-1/2} dx_1 \cdots dx_d$ on $(-1,1)^d$ (see Section~\ref{sec:application} in this paper).
Although the product rule can respond to the preference for simplicity of construction, it has the serious drawback that the size of the resulting weighted design grows exponentially. 
In this paper, we provide a method for reducing the size of designs obtained through the product rule. 

A main purpose of this paper is to develop a generalization of both Victoir's method~\cite{V04} and its variant by the present authors \cite{NS13}
in the framework of Euclidean polynomial spaces; see Definition~\ref{def:spherical_poly} of this paper.
The Victoir method reduces the size of a weighted design by replacing specific point-subsets with structured configurations associated with combinatorial objects such as combinatorial designs or orthogonal arrays.  
Combinatorial designs can be interpreted as designs in Johnson schemes, while orthogonal arrays correspond to designs in Hamming schemes.  
These designs admit two main types of generalization:  
one is the design in association schemes, introduced by Delsarte \cite{D73};  
the other is the design in polynomial spaces, proposed by Godsil \cite{G2}.  
These two notions are slightly different, in the sense that an object satisfying one definition may not satisfy the other. 
Delsarte’s designs are based on the representation theory of finite groups, whereas Godsil’s designs focus on approximating the average of polynomial functions over the entire set.  
Since our goal is to reduce the size of weighted designs, Godsil’s framework is quite effective in our context.

In design theory, a major recent breakthrough is the explicit construction of spherical $t$-designs in $\mathbb{R}^d$ for arbitrary values of $d$ and $t$ \cite{BNOZ2022, X22}.
As Xiang \cite{X22} also described, constructing weighted designs on a computer using approximations of real numbers with arbitrary precision is important in numerical analysis, but it cannot be regarded as an explicit construction method.
As an application of our generalization, we present an explicit construction of equi-weighted $5$-designs with $O(d^4)$ points for Gaussian measure or equilibrium measure, where $d$ is any integer at least $5$ (Theorems~\ref{thm:HK_1} and~\ref{thm:equilibrium1}). 
We also prove, in a manner not entirely constructive, an existence theorem of Gaussian $t$-designs with $N$ points for any $t \geq 2$, where $N<q^{t}d^{t-1}=O(d^{t-1})$ for fixed sufficiently large prime power $q$ (Theorem~\ref{thm:OA_1}).

This paper is organized as follows.  
In Section~\ref{sec:poly_sp}, we review the precise definition of designs for polynomial spaces and some fundamental results.  
We slightly generalize the class of spherical polynomial spaces, which include those arising from association schemes, by introducing the concept of Euclidean polynomial spaces. 
This generalization extends the range of applicable combinatorial structures, including regular $t$-wise balanced designs. 
In Section~\ref{sec:generalization_Victoir}, we generalize Victoir's method to designs for polynomial spaces (Theorem~\ref{thm:main}). 
To this end, we extend some known results from \cite{G1,G2}.  
This generalization enables the use of orthogonal arrays of arbitrary levels, beyond level 2 which is the only case treated in Victoir~\cite{V04}. 
In Section~\ref{sec:application}, we first present an explicit construction of orthogonal arrays of strength~$t$ with small run sizes by using the dual codes of certain extended BCH codes. 
We then apply Theorem \ref{thm:main} (and Corollaries~\ref{cor:main_finite} and \ref{cor:Nozaki_1}) with those orthogonal arrays to large weighted designs obtained via the product rule.
This not only leads to an explicit construction of equi-weighted $5$-designs with $O(d^4)$ points for Gaussian measure or equilibrium measure, but also establishes an existence theorem of weighted $t$-designs with relatively few points in high dimensions.
Remarkably, the latter result includes an improvement of a famous result by Milman~\cite{Milman1988} (see also~\cite{LV1993}) on isometric embeddings of the classical finite-dimensional Banach spaces (Corollary~\ref{cor:OA_1}). 
Finally, in Section~\ref{sec:concluding}, we discuss several directions for future work.

\section{Polynomial spaces and their designs}\label{sec:poly_sp}

We provide several terminologies and fundamental results for polynomial spaces, which are shown in \cite{G1,G2}. 
Let $\Omega$ be a set whose size is not necessarily finite. 
The {\it separation function} $\rho$ on $\Omega$ is a function from $\Omega \times \Omega$ to the real field $\R$. Any field is possible instead of $\R$, but we use only $\R$ in this paper. 
For given $a \in \Omega$, the function
\[
\rho(a,\xi)=\rho_a(\xi)
\]
is interpreted as that from $\Omega$ to $\R$. 
For $f \in \R[x]$, the {\it zonal polynomial} with respect to $a\in \Omega$ is $f\circ \rho_a: \Omega \rightarrow \R$. 
Let $Z(\Omega,r)$ denote the linear space spanned by all zonal polynomials of degree at most $r$: 
\[
Z(\Omega, r)={\rm Span}_{\R}\{f \circ \rho_a \mid a \in \Omega, f \in \R[x], {\rm deg}f\le r\}.
\]
Let ${\rm Pol}(\Omega,0)$ be the space of constant functions on $\Omega$ and ${\rm Pol}(\Omega,1)=Z(\Omega,1)$. 
We inductively define
\begin{align*}
    {\rm Pol}(\Omega,r+1)={\rm Span}_\R \{gh \mid g \in {\rm Pol}(\Omega,r), h \in {\rm Pol}(\Omega,1)\}.
\end{align*}
It is clear that $Z(\Omega, r)$ is a subspace of ${\rm Pol}(\Omega,r)$. 
The two spaces $Z(\Omega, r)$ and ${\rm Pol}(\Omega, r)$ are in general different. For instance, $\rho_a \rho_b$ belongs to ${\rm Pol}(\Omega, 2)$, but it may not belong to $Z(\Omega, 2)$ when $a$ and $b$ are distinct elements of $\Omega$. 
Let ${\rm Pol}(\Omega)$ denote $\bigcup_{r\geq 0}{\rm Pol}(\Omega,r)$. An element of ${\rm Pol}(\Omega)$ is called a {\it polynomial} (function) on $\Omega$. 
An element of ${\rm Pol}(\Omega,r)\setminus {\rm Pol}(\Omega,r-1)$ is called 
a {\it polynomial of degree $r$} on $\Omega$. 
The space ${\rm Pol}(\Omega,r)$ depends on the separation function $\rho$. When we would like to emphasize the dependence on the separation function $\rho$, we denote it as ${\rm Pol}_{\rho}(\Omega,r)$. 

The concept of polynomial spaces is defined by Godsil \cite{G1,G2}. 
\begin{definition}[Polynomial space]
Let $\rho$ be a separation function on a set $\Omega$. Let $\langle f,g \rangle$ be an inner product on ${\rm Pol}(\Omega)$. 
We call $(\Omega, \rho)$ a {\it polynomial space} if the following are satisfied. 
\begin{enumerate}
    \item For any $x, y \in \Omega$, $\rho(x,y)=\rho(y,x)$. 
    \item The dimension of ${\rm Pol}(\Omega,r)$ is finite for any $r$. 
    \item For any $f,g \in {\rm Pol}(\Omega)$, $\la f, g \ra=\la 1, fg \ra$.
    \item For any $f \in {\rm Pol}(\Omega)$, if $f(x)\geq 0$ for each $x \in \Omega$, then $\la 1,f\ra \geq 0$ holds. Moreover, $\la 1,f\ra = 0$ if and only if $f=0$. 
\end{enumerate}
\end{definition}
For the innper product $\langle \cdot, \cdot \rangle$, we may suppose $\langle 1, 1 \rangle=1$. We usually use the following inner product on ${\rm Pol}(\Omega)$ for a finite underlying set $\Omega$: 
\[
\la f, g \ra= \sum_{x \in \Omega}  f(x)g(x) \mu(x)
\]
for $f,g \in {\rm Pol}(\Omega)$, where $\mu(x)$ is
a probability function on $\Omega$. 
We can regard $\Omega$ as a weighted set with weight function $\mu$. If we do not suppose $\Omega$ is weighted, we always use $\mu(x)\equiv 1/|\Omega|$. 
The concept of designs in a polynomial space, to be defined later, is influenced by the choice of an inner product $\langle \cdot , \cdot \rangle$.

There exists an equivalence class of separation functions, all of which define the same space of polynomials. 
Two separation functions $\rho$ and $\sigma$ on $\Omega$ are {\it affinely equivalent} if
there exist $a \in \R\setminus \{0\}$ and $b\in \R$ such that for any $x,y \in \Omega$, $\rho(x,y)=a \sigma(x,y)+b$. 
If $\rho$ and $\sigma$ are affinely equivalent, 
then ${\rm Pol}_\rho (\Omega, r) = {\rm Pol}_\sigma (\Omega,r)$ for any $r\geq 0$.   
The {\it dimension} of a polynomial space is $\dim {\rm Pol}(\Omega, 1)-1$.  

In the spherical polynomial space defined in Definition \ref{def:spherical_poly}, any function in the space ${\rm Pol}(\Omega,r)$ can be represented in a simple form in the sense of Theorem \ref{thm:pol}. 
This concept is generalized to the Euclidean polynomial space while preserving the property described in Theorem~\ref{thm:pol}.

\begin{definition}[Euclidean and spherical polynomial space] \label{def:spherical_poly}
A polynomial space $(\Omega,\rho)$ is {\it Euclidean} (resp.\ {\it spherical}) if there exist an injection $\tau$ from $\Omega$ to the Euclidean space $\R^{d}$ (resp.\ the unit sphere $\mathbb{S}^{d-1}$) for some $d$ and a separation function $\rho'$ affinely equivalent to $\rho$, such that for any $x,y \in \Omega$, \[
\rho'(x,y)=(\tau(x),\tau(y)),\]
where $(\cdot , \cdot)$ is the usual inner product of $\R^d$. 
\end{definition}

\begin{example}
Johnson schemes and Hamming schemes are spherical polynomial spaces, where the usual distance functions can be taken as separation functions. 
More generally, a symmetric association scheme is regarded as a spherical polynomial space if the embedding determined by a primitive idempotent $E_i$ is injective. Since $E_i$ is a positive semidefinite matrix indexed by the vertex set, it serves as the Gram matrix of an embedding into $\mathbb{R}^{\operatorname{rank}(E_i)}$.
In this case, a separation function can be taken as $\rho(x,y) = (E_i)_{xy}$. 
For Johnson schemes and Hamming schemes, the usual distance functions are affinely equivalent to $\rho(x,y) = (E_i)_{xy}$.
\end{example}

\begin{example} \label{ex:2.4}
Let $\Omega$ be the power set of a $d$-point set, and define $\rho(a,b) = |a \cap b|$ for $a,b \in \Omega$. 
Then $(\Omega, \rho)$ is a Euclidean polynomial space, where the injection $\tau : \Omega \to \R^d$ is given by taking the characteristic vector of each subset.
\end{example}

 For a Euclidean polynomial space $(\Omega,\rho)$,  any polynomial of degree $r$ on $\Omega$ is expressed by an element of $Z(\Omega,r)$. 
 
\begin{theorem}[{\cite[Theorem 4.1 in Section 15.4 and the comment below its proof]{G1}, \cite{G2}}] \label{thm:Z=P}
A Euclidean polynomial space $(\Omega, \rho)$ satisfies that $Z(\Omega,r)={\rm Pol}(\Omega,r)$ for any $r \geq 0$. 
\end{theorem}

The following theorem was proved only for spherical polynomial spaces in \cite{G1}. 
This theorem is generalized for Euclidean polynomial spaces with the same proof. 

\begin{theorem}[{\cite[Theorem 4.3 in Section 15.4]{G1}, \cite{G2}}] \label{thm:pol}
Let $(\Omega,\rho)$ be a Euclidean polynomial space of dimension $n$. 
Let $\tau$ be an injection from $\Omega$ to $\R^n$ such that for any $x,y \in \Omega$, $\rho(x,y)=(\tau(x),\tau(y))$. 
Then it holds that for each $r\geq 0$, 
\[
{\rm Pol}(\Omega,r)=\{f \circ \tau :\Omega \rightarrow \R \mid f \in \R_r[x_1,\ldots, x_n]\}. 
\]
\end{theorem}

\begin{remark}
For $Q$-polynomial association schemes with respect to the ordering $E_0,E_1,\ldots, E_d$, the maximal common eigenspaces $V_i$ are identified with the polynomials of degree $i$ on $\Omega$, namely 
\[
{\rm Pol}(\Omega,r)= \bigoplus_{i=0}^r V_i,  
\]
    where $v =(v_x)_{x \in X} \in \R^X$ is interpreted as the function $x \in X \mapsto v_x \in \R$. 
    In particular, the dimension of a $Q$-polynomial scheme as polynomial space is $m_1={\rm Rank} (E_1)$.  
\end{remark}

For a polynomial space $(\Omega,\rho)$, an injection $\tau: \Omega \rightarrow \R^{n}$ is said to be {\it affinely compatible} with $\rho$ if there exist $a\in \R \setminus \{0\}$ and $b\in \R$ such that for any $x,y \in \Omega$, $\rho(x,y)=a(\tau(x),\tau(y))+b$. 
If there exists an affinely compatible injection $\tau$, then separation functions $\rho$ and $\sigma(x,y)=(1/a)\rho(x,y) -b/a=(\tau(x),\tau(y))$ are affinely equivalent, and hence ${\rm Pol}_\rho(\Omega,r)={\rm Pol}_\sigma(\Omega,r)$. 
From Theorem \ref{thm:pol}, the following corollary follows. 

\begin{corollary} \label{rem:com}
 Let $(\Omega,\rho)$ be a Euclidean polynomial space of dimension $n$. Suppose there exists an injection $\tau: \Omega \rightarrow \R^n$ that is affinely compatible with $\rho$. Then it holds that for each $r\geq 0$, 
\[
{\rm Pol}(\Omega,r)=\{f \circ \tau :\Omega \rightarrow \R \mid f \in \R_r[x_1,\ldots, x_n]\}. 
\]
\end{corollary}

Designs in polynomial spaces, which will play an essential role, are defined as follows. 
\begin{definition}
[Weighted $t$-design]
\label{def:t-design}
Let $(\Omega,\rho)$ be a polynomial space. 
 A finite subset $X$ of $\Omega$ is called a {\it weighted $t$-design} if 
 there exists a positive weight function $w:X\rightarrow \R_{>0}$ such that 
 for any $f \in {\rm Pol}(\Omega,t)$,  
 \begin{equation} \label{eq:def_design}
 \la 1, f\ra = \sum_{x\in X} w(x) f(x).    
 \end{equation}
 If $w$ is constant ($w\equiv 1/|X|$ under the assumption $\langle 1,1 \rangle=1$), then a weighted $t$-design is called a {\it $t$-design}. 
\end{definition}
For a finite underlying set $\Omega$, equation \eqref{eq:def_design} forms 
\[
\sum_{x \in \Omega}  f(x)\mu(x) =\sum_{x \in X} w(x)f(x),  
\]
where $\sum_{x \in \Omega}\mu(x)= \sum_{x \in X} w(x)=1$. 

\begin{example}[Spherical design]
Let $\Omega$ be the $(d-1)$-dimensional unit sphere
$\mathbb{S}^{d-1}$, 
and separation function $\rho$ be defined as the usual inner product of $\R^d$. Then, $(\Omega,\rho)$ is a spherical polynomial space and ${\rm Pol}(\Omega,r)$ is the space of polynomial functions on
$\mathbb{S}^{d-1}$
 of total degree at most $r$.  
If the inner product of ${\rm Pol}(\Omega)$ is defined as the usual normalized inner product on the sphere
\[
\langle f,g \rangle =
\frac{1}{|\mathbb{S}^{d-1}|}\int_{\mathbb{S}^{d-1}} f(x) g(x)\, d\nu(x), 
\]
where $\nu$ is the uniform measure on $\mathbb{S}^{d-1}$, then the $t$-designs in the sense of Defnition \ref{def:t-design} coincide with spherical $t$-designs \cite{DGS77}. 
\end{example} 

\begin{example}[Orthogonal array \cite{HSSbook}]
Let $S$ be a finite set of $s$ symbols. 
An $N\times k$ array $A$ with entries from $S$ is called an {\it orthogonal array with $s$ levels and strength $t$} 
if there exists $\lambda \in \mathbb{Z}$ such that
every $N \times t$ subarray of $A$ contains each $t$-tuple from $S$ exactly $\lambda$ times as a row. 
Such an orthogonal array is denoted by $\mathrm{OA}(N,k,s,t)$. 
A row in an orthogonal array is called a {\it run}. 
The set of runs is a subset of $\Omega=S^k$. 
   The set of runs in an orthogonal array of strength $t$ is identified with a $t$-design in Hamming schemes $H(k,s)$, where the probability function $\mu(x)$ is constant~\cite{G1,G2}. 
\end{example}

\begin{example}[Combinatorial $t$-design \cite{BJLbook}]
Let $\Omega$ be the set of $k$-subsets of a $v$-point set $V$, and let $\mathcal{B}$ be a subset of $\Omega$, whose elements are called \emph{blocks}. 
The pair $(V, \mathcal{B})$ is called a \emph{$t$-$(v,k,\lambda)$ design} (or a \emph{combinatorial $t$-design}) if every $t$-subset of $V$ is contained in exactly $\lambda$ blocks. 
The block set of a combinatorial $t$-design is identified with a $t$-design in the Johnson scheme $J(v,k)$, where the probability function $\mu(x)$ is constant \cite{G1,G2}. 
\end{example}

\begin{example}[Regular $t$-wise balanced design \cite{FKJ89}] \label{ex:t-wise}
Let $\Omega$ be the power set of a $v$-point set $V$, and let $\mathcal{B}$ be a subset of $\Omega$, whose elements are called \emph{blocks}.  
The block set $\mathcal{B}$ is partitioned by size as $\mathcal{B} = \bigcup_{i=1}^\ell \mathcal{B}_i$. 
Let $K = \{k_1, \ldots, k_\ell\}$ be the set of block sizes, where $k_i = |x|$ for all $x \in \mathcal{B}_i$. 
The pair $(V, \mathcal{B})$ is called a \emph{regular $t$-$(v,K,\lambda)$ design} (or a \emph{regular $t$-wise balanced design}) if for each integer $t'$ with $0 \le t' \le t$, every $t'$-subset of $V$ is contained in exactly $\lambda_{t'}$ blocks, where $\lambda = \lambda_t$. 
The block set of a regular $t$-$(v,K,\lambda)$ design is identified with a $t$-design in $(\Omega, \rho)$ (Example~\ref{ex:2.4}) with probability function
\[
\mu(x) = \begin{cases}
\frac{|\mathcal{B}_i|}{|\mathcal{B}| \binom{v}{k_i}} &\text{if $|x|=k_i$}, \\
0 &\text{otherwise}
\end{cases}
\]
for $x \in \Omega$ \cite[Proposition~3.1]{NS13}.
\end{example}

Corollary~\ref{rem:com} implies the following theorem. 

\begin{theorem}\label{thm:1}
 Let $(\Omega,\rho)$ be a Euclidean polynomial space of dimension $n$. 
Let $X \subset \Omega$ be a weighted $t$-design in $\Omega$ with weight function $w$. 
Suppose there exists an injection $\tau: \Omega \rightarrow \R^{n}$ that is affinely compatible with $\rho$. Then it holds that 
\[
\la 1,f\circ \tau \ra 
=  \sum_{x \in X} w(x) f\circ \tau(x)
\]   
for any $f \in \R_t[x_1,\ldots, x_{n}]$. 
In particular, if $\Omega$ is a finite weighted set with probability function $\mu$, 
then
\[
  \sum_{x \in \Omega} f\circ \tau(x)  \mu(x)
=  \sum_{x \in X} w(x)f\circ \tau(x), 
\]  
and hence
\[
 \sum_{y \in \tau(\Omega)}  f(y)\mu(\tau^{-1}(y))
=  \sum_{y \in \tau(X)} w(\tau^{-1}(y))f(y)
\]
for any $f \in \R_t[x_1,\ldots, x_{n}]$. 
\end{theorem}

\section{Generalization of Victoir's method} \label{sec:generalization_Victoir}

Victoir \cite{V04} proposed a method for reducing the size of a weighted $t$-design $Z$ in $\mathbb{R}^d$ that contains a specific subset $X$, by replacing $X$ with another set $Y$ associated with a combinatorial structure, while preserving the $t$-design property. 
He used combinatorial designs, orthogonal arrays with 2 levels, and $t$-homogeneous sets as sources of such replacement sets $Y$, so that the resulting weighted design $(Z \setminus X) \cup Y$ has fewer points.
In this section, we generalize his method by using designs in Euclidean polynomial spaces $\Omega$.
Here, $X$ is regarded as $\tau(\Omega)$ for some affinely compatible injection $\tau : \Omega \rightarrow \mathbb{R}^d$ with separation function $\rho$,  and $Y$ is taken as the image of a design in $\Omega$.
A different type of generalization is required for each of the three combinatorial structures.

\subsection{Orthogonal array and Hamming scheme with 2 levels}

In Victoir's method, the set $X = \{\pm 1\}^d$ is replaced by a set $Y$ associated with an orthogonal array $\mathrm{OA}(N,d,2,t)$ in order to reduce the size of a weighted $t$-design in $\mathbb{R}^d$. 
The Hamming scheme $\Omega = H(d,2)$, equipped with the Hamming distance $\rho$, forms a spherical polynomial space, where an orthogonal array is regarded as a $t$-design in the sense of Definition~\ref{def:t-design}. 
The underlying set $\Omega$ is $\{0,1\}^d$. 
The natural bijection $\tau: \{0,1\}^d \to \{-1,1\}^d\subset \R^d$, where $0$ is mapped to $-1$ and $1$ to $1$, is affinely compatible with $\rho$.
Since the dimension of $(\Omega, \rho)$ is $d$, the image $\tau(\Omega) = X$ can be replaced by the image of an orthogonal array while preserving the $t$-design property, as guaranteed by Theorem~\ref{thm:1}.
Thus, Victoir's method for orthogonal arrays with 2 levels follows directly from Theorem~\ref{thm:1}.

\subsection{Combinatorial design and Johnson scheme}
\label{subsect:Johnson}
Let $v_k(\alpha, \beta) \subset \mathbb{R}^d$ be the set consisting of all vectors whose $k$ coordinates are equal to $\alpha$, and the remaining $d - k$ coordinates are equal to $\beta$, where $\alpha \neq \beta$.
In Victoir's method, the set $X=v_k(\alpha, \beta)$ is replaced by a set $Y$ associated with a combinatorial $t$-design in order to reduce the size of a weighted $t$-design.
The Johnson scheme $\Omega = J(d,k)$, equipped with the usual distance
$\rho(x,y) = 2(k - |x \cap y|)$, 
forms a spherical polynomial space, where a combinatorial design is regarded as a $t$-design in the sense of Definition~\ref{def:t-design}.
The underlying set $\Omega$ is the set of all $k$-subsets of $\{1,2,\ldots,d\}$.
The bijection $\tau: \Omega \to v_k(\alpha, \beta)\subset \R^{d}$, defined by assigning to each $T \in \Omega$ the vector $(x_1, \ldots, x_d)$ with $x_i = \alpha$ if $i \in T$ and $x_i = \beta$ if $i \notin T$, is affinely compatible with $\rho$. 
In Victoir's method, $\tau(\Omega)=X$ is replaced by the image of a combinatorial design while preserving the $t$-design property. 

The dimension of $(\Omega, \rho)$ as a polynomial space is $d - 1$, which is smaller than the dimension of the range of $\tau$.
Theorem~\ref{thm:1} applies only to the case where the dimension of the range of $\tau$ coincides with that of the polynomial space $(\Omega, \rho)$.
In order to treat more general cases, we need to extend Corollary~\ref{rem:com} to the setting of any injection $\Omega \to \mathbb{R}^m$, where $m$ is at least the dimension of the polynomial space.

\begin{theorem} \label{thm:key}
Let $(\Omega,\rho)$ be a Euclidean polynomial space of dimension $n$. 
Suppose there exists an injection $\tau :\Omega \rightarrow \R^{m}$ that is affinely compatible with $\rho$, and $m\geq n$. 
Then it holds that for each $r\geq 0$, 
\[
{\rm Pol}(\Omega,r)=\{f \circ \tau :\Omega \rightarrow \R \mid f \in \R_r[x_1,\ldots, x_m]\}. 
\]
\end{theorem}
\begin{proof}
It is sufficient to prove the case $r=1$. For $r>1$, we can prove this theorem inductively by definition.  
Since the dimension of $(\Omega,\rho)$ is $n$, we have $n=\dim {\rm Pol}(\Omega,1)-1$. Since $\tau$ is affinely compatible with $\rho$,  by ${\rm Pol}(\Omega,1)=Z(\Omega,1)$, 
\begin{align*}
{\rm Pol}(\Omega,1)&={\rm Span}_\R \{f\circ \rho_a(\xi) \mid a \in  \Omega, f \in \R_1[x]\}\\
&={\rm Span}_\R\{ 1 \} + {\rm Span}_\R \{ \rho_a(\xi) \mid a \in  \Omega\}\\
&={\rm Span}_\R\{ 1 \} + {\rm Span}_\R \{ (\tau(a),\tau(\xi)) \mid a \in  \Omega\},
    \end{align*}
and hence the dimension of $V= {\rm Span}_\R \{ (\tau(a),\tau(\xi)) \mid a \in  \Omega\}$ is $n$ or $n+1$. 

It is sufficient to prove that 
\begin{equation}
    \label{eq:1}
    \{f \circ \tau :\Omega \rightarrow \R \mid f \in \R[x_1,\ldots, x_m], \deg f = 1, \text{$f$ is homogeneous}\}\subset V,
\end{equation}
because the reverse inclusion is clear. 
We need to prove that for each $i\in \{1,\ldots,m\}$, $x_i \circ \tau(\xi) =(e_i,\tau(\xi))$ belongs to $V$, where $\{e_i\}$ is the standard basis of $\mathbb{R}^m$.
From the linearity of the inner product, it suffices to prove that $(v_i,\tau(\xi)) \in V$ with some basis $\{v_i \}$ of $\R^m$. 

Let $n'=\dim V$. 
For $m=n$, the theorem is trivial, and we suppose $m>n$, that is, $m\geq n'$.  
Since the maximum rank of the Gram matrices  $[(\tau(x),\tau(y))]_{x,y \in X}$ over all finite subsets $X\subset \Omega$ is $n'$, we have $\dim {\rm Span}_\R\tau(\Omega)=n'$. 
Let $\{v_1,\ldots, v_{n'}\}$ be a basis of ${\rm Span}_\R \tau(\Omega)$. 
We can construct a basis $\{v_1,\ldots, v_m\}$ of $\R^m$ containing $\{v_1,\ldots, v_{n'} \}$ such that $v_k\perp \{v_1,\ldots, v_{n'} \}$ for each $k>n'$. 

For each $i\in\{1,\ldots,n' \}$, there exist $\lambda_a \in \R$ ($a \in \Omega$) such that 
\[
v_i=\sum_{a \in \Omega}\lambda_a \tau(a).
\]
By the linearity of the inner product, 
\[
(v_i,\tau(\xi)) = \sum_{a \in \Omega} \lambda_a (\tau(a),\tau(\xi)) \in V. 
\]
For any $i\in\{n'+1,\ldots, m \}$ and $a \in \Omega$, 
we have $(v_i,\tau(a))=0$, and hence 
\[
(v_i,\tau(\xi))=0 \in V
\]
on $\Omega$. Therefore, $\eqref{eq:1}$ holds, and this theorem follows. 
\end{proof}

Theorem \ref{thm:key} implies the following theorem that generalizes Victoir's method on combinatorial designs. 

\begin{theorem} \label{thm:m>=n}
 Let $(\Omega,\rho)$ be a Euclidean polynomial space of dimension $n$. 
Let $X \subset \Omega$ be a weighted $t$-design in $\Omega$ with  weight function $w$. 
Suppose there exists an injection $\tau: \Omega \rightarrow \R^{m}$ that is affinely compatible with $\rho$ for some $m\geq n$. Then it holds that 
\[
\la 1,f\circ \tau \ra 
=  \sum_{x \in X} w(x) f\circ \tau(x)
\]   
for any $f \in \R_t[x_1,\ldots, x_{m}]$. 
In particular, if $\Omega$ is a finite weighted set with probability function $\mu$, 
then
\[
  \sum_{x \in \Omega} f\circ \tau(x) \mu(x) 
=  \sum_{x \in X} w(x)f\circ \tau(x), 
\]  
and hence
\[
 \sum_{y \in \tau(\Omega)}  f(y) \mu(\tau^{-1}(y))
=  \sum_{y \in \tau(X)} w(\tau^{-1}(y))f(y)
\]
for any $f \in \R_t[x_1,\ldots, x_{m}]$. 
\end{theorem}
\begin{proof}
    It follows from Theorem \ref{thm:key} that $f\circ \tau \in {\rm Pol}(\Omega, t)$. 
    This theorem is clear by the definition of design. 
\end{proof}

\subsection{$t$-transitive group and symmetric group}

Initially, we provide basic terminology. Let $\mathfrak{S}_d$ denote the symmetric group of degree $d$. 
Let $[d]^{\{k\}}$ be the set of $k$-point subsets of $[d]=\{1,\ldots,d\}$. The group $\mathfrak{S}_d$ naturally acts on $[d]^{\{k\}}$. 
For $A,B \in [d]^{\{k\}}$, $[A;B]$ denotes the subset of $\mathfrak{S}_d$ which consists of all permutations that move $A$ to $B$. A non-empty subset $Z$ of $\mathfrak{S}_d$ is {\it $k$-homogeneous} if $Z \cap [A;B]$ is independent of $A,B \in [d]^{\{k\}}$ with $A\ne B$. 
It is known that a $k$-homogeneous set is $(k-1)$-homogeneous for $2\leq k\leq d/2$ \cite{N85}. 
A $k$-homogeneous set is a $(d-k)$-homogeneous set by considering the complement of $A\in [d]^{\{k\}}$. A subgroup $H$ of $\mathfrak{S}_d$ is {\it $k$-homogeneous} if $H$  acts transitively on $[d]^{\{k\}}$. A $k$-homogeneous group is a $k$-homogeneous set. 
Let $[d]^{(k)}$ be the set of ordered $k$-tuples of distinct elements of $[d]$. 
A subgroup $H$ of $\mathfrak{S}_d$ is {\it $k$-transitive} if $H$ acts transitively on $[d]^{(k)}$. 
A $k$-transitive subgroup is $(k-1)$-transitive. 

For $x=(x_1,\ldots, x_d)\in \mathbb{R}^d$ and a subset $Z$ of $\mathfrak{S}_d$, define
\[
Z\cdot x=\{(x_{\sigma(1)},\ldots, x_{\sigma(d)}) \mid \sigma \in Z\}.
\]
If $Z$ is a subgroup, then $Z \cdot x$ is the orbit of $x$ under the action of permutations of coordinates. 
Theorem 3.5 in Victoir \cite{V04} asserts that a $t$-homogeneous set $Z\subset \mathfrak{S}_d$ defines a weighted $t$-design, that is, 
\begin{equation*} 
\frac{1}{|\mathfrak{S}_d\cdot x|} \sum_{y \in \mathfrak{S}_d \cdot x} f(y)=\frac{1}{|Z\cdot x|} \sum_{y \in Z \cdot x} f(y)    
\end{equation*}
for any polynomial $f$ of degree at most $t$. Victoir omitted the proof of this theorem, merely stating to check it for the monomials. 
Unfortunately, there exists a counterexample to this theorem, and the assertion requires modification. 
Kantor \cite{K72} investigated a $k$-homogeneous group but not $k$-transitive for $d\geq 2k$. One example is $PSL(2,8) \subset S_9$, which is $4$-homogeneous but only $3$-transitive. Indeed, $PSL(2,8)$ is 9-homogeneous, but it does not satisfy \eqref{eq:vicdes} for some polynomials of degree 5. 
We should modify Victoir's theorem as follows. 

\begin{theorem} \label{thm:imp_vic}
Suppose $Z\subset \mathfrak{S}_d$ is a $(t-1)$-transitive group which is $t$-homogeneous. 
Then it holds that 
\begin{equation} \label{eq:vicdes}
\frac{1}{|\mathfrak{S}_d\cdot x|} \sum_{y \in \mathfrak{S}_d \cdot x} f(y)=\frac{1}{|Z\cdot x|} \sum_{y \in Z \cdot x} f(y)    
\end{equation}
for any $x\in \mathbb{R}^d$ and any $f \in \mathbb{R}_t[y_1,\ldots,y_d]$. 
\end{theorem}

\begin{proof}
First, we prove that for a given point $x\in \mathbb{R}^d$ and a monomial $f(y) = y_{i_1}^{\lambda_1} \cdots y_{i_k}^{\lambda_k}$, the value
\[
\frac{1}{|Z\cdot x|} \sum_{y \in Z \cdot x} y_{i_1}^{\lambda_1} \cdots y_{i_k}^{\lambda_k}
\]
is constant for every $k$-transitive group $Z \subset \mathfrak{S}_d$. Indeed, we have
\begin{align*}
\frac{1}{|Z\cdot x|} \sum_{y \in Z \cdot x} y_{i_1}^{\lambda_1}\cdots y_{i_k}^{\lambda_k}
&= \frac{1}{|Z\cdot x|\cdot |{\rm Stab}_Z(x)|} \sum_{\sigma \in Z} x_{\sigma(i_1)}^{\lambda_1}\cdots x_{\sigma(i_k)}^{\lambda_k} \\
&= \frac{|{\rm Stab}_Z(i_1,\ldots,i_k)|}{|Z|} \sum_{(j_1,\ldots,j_k) \in Z\cdot (i_1,\ldots,i_k)} x_{j_1}^{\lambda_1}\cdots x_{j_k}^{\lambda_k} \\
&= \frac{1}{|Z\cdot (i_1,\ldots,i_k)|} \sum_{(j_1,\ldots,j_k) \in Z\cdot (i_1,\ldots,i_k)} x_{j_1}^{\lambda_1}\cdots x_{j_k}^{\lambda_k} \\
&= \frac{1}{|[d]^{(k)}|} \sum_{(j_1,\ldots,j_k) \in [d]^{(k)}} x_{j_1}^{\lambda_1}\cdots x_{j_k}^{\lambda_k},
\end{align*}
where the last equality follows from the $k$-transitivity of $Z$.

For $1\leq k \leq t-1$, 
the equality \eqref{eq:vicdes} holds for all monomials $y_{i_1}^{\lambda_1} \cdots y_{i_k}^{\lambda_k}$ in $\mathbb{R}_t[y_1,\ldots,y_d]$, since both $\mathfrak{S}_d$ and $Z$ are $k$-transitive.  
For $k=t$, the monomials in $\mathbb{R}_t[y_1,\ldots, y_d]$ are of the form $y_{i_1} \cdots y_{i_t}$. 
By the $t$-homogeneity of $Z$, we similarly have
\[
\frac{1}{|Z\cdot x|} \sum_{y \in Z \cdot x} y_{i_1}\cdots y_{i_t} = \frac{1}{|[d]^{\{t\}}|} \sum_{(j_1,\ldots,j_t) \in [d]^{\{t\}}} x_{j_1}\cdots x_{j_t}.
\]
Thus, the equality \eqref{eq:vicdes} holds for any $f\in \mathbb{R}_t[y_1,\ldots,y_d]$. 
\end{proof}

\begin{remark} \label{rem:transitive}
A $k$-homogeneous group is a $(k-1)$-transitive group for $d\geq 2k$ \cite{LW65}. 
Moreover, if $k\geq 5$ and $d\geq 2k$, then a $k$-homogeneous group is $k$-transitive \cite{LW65}. 
\end{remark}

The following corollary is immediate. 
\begin{corollary} \label{cor:t-transitive}
     If $Z\subset \mathfrak{S}_d$ is a $t$-transitive group, 
    then $Z$ satisfies \eqref{eq:vicdes} for any polynomial of degree at most $t$. 
\end{corollary}

We would like to understand Corollary  \ref{cor:t-transitive} within the framework of polynomial space. 
Let $\rho(x,y)$ be the number of points fixed by $x^{-1}y$ for $x,y \in \mathfrak{S}_d$.  
The symmetric group $\Omega=\mathfrak{S}_d$ forms a spherical polynomial space with the separation function $\rho$. 
When a subset $X$ of $\Omega$ is a subgroup, it holds that $X$ is a $t$-design if and only if $X$ is a $t$-transitive group~\cite{G2}. 
The symmetric group $\mathfrak{S}_d$ 
has the structure of a commutative association scheme as a group scheme \cite[p.\ 54, Example 2.1(2)]{BIb}. A $t$-transitive group is also a $T$-design in the sense of Delsarte \cite{D73}, where $T$ consists of the set of primitive idempotents corresponding to the irreducible representations whose Young diagrams have level at most $t$ \cite[Equation (3.1)]{B84}. 
In particular, the primitive idempotent $E$ that defines a 1-design (as a design in the polynomial space) corresponds to the Young diagram $(d-1,1)$. 
The matrix $E$ is the Gram matrix of a spherical injection $\tau$, and hence the dimension of the polynomial space $(\Omega, \rho)$ is equal to the rank of $E$, which is $(d-1)^2$.  
This dimension $(d-1)^2$ is larger than the dimension $d$ of a weighted design in \eqref{eq:vicdes}. Thus, Corollary \ref{cor:t-transitive} cannot be derived as a consequence of Theorem \ref{thm:m>=n}. 
As a further generalization of Theorem \ref{thm:m>=n}, we prove that if the design $X$ can be embedded into lower dimensions via an affine map, then the design property of $X$ is preserved.  

\begin{theorem} \label{thm:main}
Let $(\Omega,\rho)$ be a Euclidean polynomial space of dimension $n$. Let $X$ be a weighted $t$-design in $(\Omega, \rho)$ with weight function $w$.  
For $m\geq n$, let $\tau: \Omega \rightarrow \R^{m}$ be an injection affinely compatible with $\rho$. 
Let $\xi: \R^m \rightarrow \R^s$ be an affine map.  
Then,  $f \circ \xi \circ \tau \in {\rm Pol}(\Omega,t)$, and 
\[
(1,f\circ \xi \circ \tau)=
 \sum_{x \in X} w(x) f\circ \xi \circ \tau(x)
\]
for each $f \in \R_t[x_1,\ldots,x_s]$.
\end{theorem}

\begin{proof}
Note that $f \circ \xi \in \R_t[x_1,\ldots,x_m]$ for each $f \in \R_t[x_1,\ldots, x_s]$. It follows from Theorem~\ref{thm:key} that $(f \circ \xi) \circ \tau \in {\rm Pol}(\Omega,t)$. 
The assertion is clear from the definition of design. 
\end{proof}

It is noteworthy that an affine map $\xi$ for Theorem \ref{thm:main} is not required to be affinely compatible with $\rho$, and in such a case, $\xi$ is not necessarily injective. 

\begin{remark}
The present authors proposed a variant of Victoir's method using regular $t$-$(d,K,\lambda)$ designs \cite{NS13}. 
Theorem \ref{thm:main} applies to this case.
The corresponding Euclidean polynomial space $(\Omega, \rho)$ is defined in Example \ref{ex:2.4}. 
The dimension of $(\Omega, \rho)$ is $d$ when $|K|>1$. 
The injection $\tau: \Omega \rightarrow \{1,0\}^d \subset \mathbb{R}^d$, defined by the characteristic vectors, is affinely compatible with $\rho$. 
The injective affine map $\xi:\{1,0\}^d \rightarrow \{\alpha,
\beta\}^d$ is defined by $x \in \{1,0\}^d \mapsto (\alpha -\beta )x+ \beta j$, where $j$ is the all-ones vector. 
Considering the probability function $\mu$ defined in Example~\ref{ex:t-wise}, 
the specific weighted subset $X$ of a weighted $t$-design $Z$ in $\mathbb{R}^d$ 
is taken as $\bigcup_{k \in K} v_k(\alpha, \beta)$, with weight function $\mu(x)$. 
For the subset $Y \subset X$ corresponding to a regular $t$-$(d, K, \lambda)$ design, 
$X$ can be replaced by $Y$ as
\[
\sum_{x \in X} f(x) \mu(x) = \frac{1}{|Y|} \sum_{y \in Y} f(y),
\]
while preserving the $t$-design property.
\end{remark}

\begin{corollary} \label{cor:main_finite}
With the same setup as in Theorem \ref{thm:main},  we moreover assume that $\Omega$ is finite, $\mu \equiv 1/|\Omega|$, and $w\equiv 1/|X|$.   
Let $\Omega'=\xi\circ \tau(\Omega)$ and $X'=\xi\circ \tau(X)$. 
Assume $|(\xi\circ \tau)^{-1}(x)|$ is constant for all $x \in \Omega'$ and $|X \cap (\xi\circ \tau)^{-1}(x)|$ is constant for all $x \in X'$. Then it holds that 
\[
\frac{1}{|\Omega'|} \sum_{x \in \Omega'} f(x)=
\frac{1}{|X'|} \sum_{x \in X'} f(x)
\]
for each $f \in \R_t[x_1,\ldots,x_s]$. 
\end{corollary}

\begin{proof}
    From our assumption, $|\Omega|=|(\xi\circ \tau)^{-1}(x)|\cdot |\Omega'|$ for each $x \in \Omega'$ and $|X|=|X\cap (\xi\circ \tau)^{-1}(x)|\cdot |X'|$ for each $x \in X'$. 
    Therefore, it follows that for each $f \in \R_t[x_1,\ldots, x_s]$, 
 \begin{align*}
    \frac{1}{|\Omega'|} \sum_{x \in \Omega'} f(x)
    & =\frac{1}{|\Omega|} \sum_{x \in \Omega'} |(\xi\circ \tau)^{-1}(x)|\cdot f(x) \\
    & =\frac{1}{|\Omega|} \sum_{y \in \Omega} f( \xi \circ \tau(y)) \\
    &=(1,f \circ \xi \circ \tau ) \\
    & = \frac{1}{|X|} \sum_{y \in X} f\circ \xi \circ \tau(y)\\
    & = \frac{1}{|X|} \sum_{y \in X} f( \xi \circ \tau(y)) \\
    & =\frac{1}{|X|} \sum_{x \in X'} |X\cap (\xi\circ \tau)^{-1}(x)|\cdot f(x) \\
    &= \frac{1}{|X'|} \sum_{x \in X'} f(x), 
 \end{align*}
where Theorem \ref{thm:main} is used in the fourth equality, which completes the proof.
\end{proof}

\begin{corollary} \label{cor:Nozaki_1}
With the same setup as in Theorem \ref{thm:main},  we moreover assume that $\Omega$ is a weighted finite set with
probability function $\mu$.  
We regard $\tilde{\Omega}=\xi\circ \tau(\Omega)$ as a multiset with multiplicity $|(\xi\circ \tau)^{-1}(x)|$ and $|\Omega|=|\tilde{\Omega}|$, and $\tilde{X}=\xi\circ \tau(X)$ is treated similarly.
Consider the function $(\xi \circ \tau)|_{\Omega}:\Omega \rightarrow \tilde{\Omega}$ as injective and write $y_x=((\xi\circ \tau)|_{\Omega})^{-1}(x)$. 
Then it holds that
\[
\sum_{x \in \tilde{\Omega}} f(x)  \mu(y_x)=
 \sum_{x \in \tilde{X}} w(y_x)  f(x)
\]
for each $f \in \R_t[x_1,\ldots,x_s]$. 
\end{corollary}

\begin{proof}
    It follows that for each $f \in \R_t[x_1,\ldots , x_s]$,
\begin{align*}
 \sum_{x \in \tilde{\Omega}} f(x) \mu(y_x)
  & = \sum_{y \in \Omega} f( \xi \circ \tau(y)) \mu(y) \\
  & = (1,f \circ \xi \circ \tau ) \\
  & = \sum_{y \in X}w(y) f\circ \xi \circ \tau(y) \\
  & = \sum_{y \in X} w(y)f( \xi \circ \tau(y)) \\
  & = \sum_{x \in \tilde{X}} w(y_x) f(x), 
\end{align*}
where we use Theorem \ref{thm:main} in the third equality, 
which completes the proof.
\end{proof}

\begin{remark}
 We demonstrate how to provide a weighted design of $\eqref{eq:vicdes}$ by Corollary~\ref{cor:main_finite}. 
 Since a $t$-transitive group is a $t$-design in $(\mathfrak{S}_d,\rho)$,
 we aim to prove Corollary \ref{cor:t-transitive}. 
We use the injection 
$\tau: \mathfrak{S}_d \rightarrow  \R^{d^2-d} \simeq \{x \in \R^d \mid \sum_i x_i=1\}^d$ defined by
\[
\tau(\sigma) =(e_{\sigma(1)},\ldots, e_{\sigma(d)}),
\]
where $\{e_i\}$ is the standard basis of $\R^d$. 
The inner product of $\tau(\sigma)$ and $(e_i,e_i,\ldots e_i)$ is 1 for each $i \in \{1,\ldots, d\}$. From this fact, the codomain can be identified with $\R^{(d-1)^2}$, whose dimension coincides with the dimension of the polynomial space.  
The injection $\tau$ is affinely compatible with $\rho$. 
Letting $\xi_x:\R^{d^2}\rightarrow \R^d$ be the linear map defined by
\[
y \mapsto y \begin{pmatrix}
    x^\top &o &\cdot &o \\
    o &x^\top &\cdot &o \\
    && \ddots & \\
    o &o &\cdot & x^\top 
\end{pmatrix} \text{ with } x=(
    x_1,  x_2, \ldots ,x_d),
\]
one has 
\[
\xi_x(e_{\sigma(1)},\ldots, e_{\sigma(d)})=(
    x_{\sigma(1)},\ldots, x_{\sigma(d)})=\sigma \cdot x. 
\]
Therefore, we have $\xi_x \circ \tau(\mathfrak{S}_d) = \mathfrak{S}_d \cdot x$. 
Moreover, for any $\sigma \in \mathfrak{S}_d$, the preimage 
$(\xi_x \circ \tau)^{-1}(\sigma \cdot x)$ coincides with the left coset 
$\sigma \cdot {\rm Stab}_{\mathfrak{S}_d}(x)$, whose cardinality is independent of $\sigma$. 
Now, let $Z$ be a $t$-transitive subgroup of $\mathfrak{S}_d$. 
Then it holds that $\xi_x \circ \tau(Z) = Z \cdot x$, and for every $\sigma \in Z$, 
the size of the intersection $Z \cap (\xi_x \circ \tau)^{-1}(\sigma \cdot x)$ equals 
$|\sigma \cdot {\rm Stab}_Z(x)|$, which is constant as well. 
Therefore, Corollary~\ref{cor:main_finite} applies to this setting, and 
Corollary~\ref{cor:t-transitive} follows.
\end{remark}

\subsection{Orthogonal array and Hamming scheme with $q$ levels}
\label{sec:OA_any_level}

As an application of Corollary \ref{cor:main_finite}, we generalize Victoir's method on the Hamming scheme with $2$ levels to any $q$ levels.
Let $\Omega=[q]^d$ and $\rho$ be the Hamming distance.
The Hamming scheme $H(d,q)$ is identified with the spherical polynomial space $(\Omega,\rho)$ of dimension $d(q-1)$. 
A $t$-design in $\Omega$ is identified with an orthogonal array $\mathrm{OA}(N,d,q,t)$. 
We use the injection  $\tau: \Omega \rightarrow \R^{d(q-1)} \simeq \{x \in \R^q \mid \sum_i x_i=1\}^d$ defined by
\[\tau(i_1,\ldots, i_d)=(e_{i_1},\ldots, e_{i_d}),\]
which is affinely compatible with $\rho$. 
For $x=(x_1,\ldots,x_q) \in \R^q$, 
define $\xi_x: \R^{dq} \rightarrow \R^d$ be the linear map defined by
\[
y \mapsto y \begin{pmatrix}
    x^\top &o &\cdot &o \\
    o &x^\top &\cdot &o \\
    && \ddots & \\
    o &o &\cdot & x^\top 
\end{pmatrix} \text{ with } x=(
    x_1,  x_2, \ldots ,x_q),
\]
and one has
\[
\xi_x(e_{i_1},\ldots, e_{i_d})=(x_{i_1},\ldots, x_{i_d}). 
\]

Suppose $x_1,\ldots, x_q$ are mutually distinct. 
Then, $\Omega'=\xi_x \circ \tau(\Omega)=\{x_1,\ldots, x_q\}^d$ and the restriction $\xi_x|_{\tau(\Omega)}$ is injective, in particular $|(\xi_x\circ \tau)^{-1} (y)|=1$ for each $y \in \Omega'$.  
Let $X$ be a $t$-design in $\Omega$, that is, the set of runs of $\mathrm{OA}(|X|,d,q,t)$. 
Let $X'=\xi_x \circ \tau(X)$. 
Since $\xi_x|_{\tau(\Omega)}$ is injective, 
$|X\cap (\xi_x\circ \tau)^{-1} (y)|=1$ for each $y \in X'$. Therefore, we can apply 
Corollary \ref{cor:main_finite}, 
and for each $f\in \R_t[x_1,\ldots, x_d]$, 
\[
\frac{1}{|\Omega'|}\sum_{x \in \Omega'} f(x)=\frac{1}{|X'|}
\sum_{x \in X'} f(x)
\]
is satisfied.

\section{Applications for constructing small weighted designs} \label{sec:application}

We present applications of a generalization of Victoir’s method (Theorem \ref{thm:main}, Corollaries~\ref{cor:main_finite} and \ref{cor:Nozaki_1}), as described in Section~\ref{sec:generalization_Victoir}.
In the examples provided later, we reduce the size of large weighted designs obtained via the product rule by applying this generalized method together with orthogonal arrays of relatively small size. 

The product rule is a method for constructing a weighted $t$-design by taking the Cartesian product of lower-dimensional weighted $t$-designs, applied to a measure that is written as the Cartesian product of lower-dimensional measures; see, e.g., \cite{HW57}.

\begin{proposition}[Product rule] \label{prop:product_1}
Let $\mathbb{R}^d = \mathbb{R}^{d_1} \times \mathbb{R}^{d_2}$, and suppose a density function $\mu(x)$ on $\mathbb{R}^d$ can be written as a product $\mu(x) = \mu_1(x_1)\mu_2(x_2)$ for $x = (x_1, x_2)$, where each $\mu_i$ is a density function on $\mathbb{R}^{d_i}$ for $i = 1, 2$.  
If $(X_i, w_i(x_i))$ is a weighted $t$-design with respect to $\mu_i$ for $i=1,2$,  
then the pair $(X_1 \times X_2, w_1(x_1)w_2(x_2))$ forms a weighted $t$-design with respect to $\mu$. 
In particular, it readily generalizes to the case of multiple Cartesian products. 
\end{proposition}
\begin{proof}
Each monomial $\xi(x)$ on $\mathbb{R}^d$ can be expressed as the product of monomials $\xi_1(x_1)$ on $\mathbb{R}^{d_1}$ and $\xi_2(x_2)$ on $\mathbb{R}^{d_2}$; that is, $\xi(x) = \xi_1(x_1)\xi_2(x_2)$ for $x = (x_1, x_2)$. 

For each monomial $\xi(x)$ of degree at most $t$, we have
\begin{align*}
    \int_{\mathbb{R}^d} \xi(x)\mu(x) \, dx 
    &= \int_{\mathbb{R}^{d_1}} \xi_1(x_1) \mu_1(x_1) \, dx_1 \cdot \int_{\mathbb{R}^{d_2}} \xi_2(x_2) \mu_2(x_2) \, dx_2 \\
    &= \sum_{x_1 \in X_1} w_1(x_1) \xi_1(x_1) \cdot \sum_{x_2 \in X_2} w_2(x_2) \xi_2(x_2) \\
    &= \sum_{(x_1, x_2) \in X_1 \times X_2} w_1(x_1) w_2(x_2) \xi_1(x_1)\xi_2(x_2) \\
    &= \sum_{x \in X_1 \times X_2} w(x) \xi(x),
\end{align*}
where $w(x) = w_1(x_1) w_2(x_2)$ for $x = (x_1, x_2)$. This shows that $(X_1 \times X_2, w(x))$ is a weighted $t$-design, as claimed. 
\end{proof}

Indeed, given a one-dimensional weighted $t$-design $(X, w(x))$, the product set $X^d$ with the weight function $(x_1, \ldots, x_d) \mapsto \prod_{i=1}^d w(x_i)$ defines a weighted $t$-design in $\mathbb{R}^d$. 
If the weight function $w(x)$ is constant, then the resulting weight in higher dimensions is also constant.  
Although this construction is simple, the size of the resulting weighted design grows exponentially with the dimension.

\subsection{Orthogonal arrays with a small number of runs}

\label{sect:makeOA}
 
In this subsection, we construct orthogonal arrays as the dual codes of extended BCH codes, which are used to reduce the size of weighted designs in the following subsections.  
First, we review fundamental concepts and key results related to orthogonal arrays and linear codes.
See \cite[Chapter 7]{MSbook} and \cite[Chapter 5]{HSSbook} for details.  

Let ${\rm GF}(q)$ be the finite field of order $q$. 
A {\it linear code} consists of the vectors of a linear subspace of ${\rm GF}(q)^k$ as codewords. 
A linear code $C$ is called a $[k, N, d]_q$ code if
$C$ is a subspace of ${\rm GF}(q)^k$ with size $N = q^{\dim C}$ and minimum Hamming distance $d$. 

The following is a well-known sufficient condition for the codewords of linear codes to be the runs of orthogonal arrays.

\begin{theorem} \label{thm:linear_code_and_OA}
    Let $A$ be an $N \times k$ array whose rows consist of the codewords of a $[k, N, d]_q$ code. 
    If any $t$ columns of $A$ are linearly independent over ${\rm GF}(q)$, then $A$ is an $\mathrm{OA}(N,k,q,t)$. 
\end{theorem}

If an orthogonal array is obtained from a linear code $C$, 
then the strength $t$ is at least the minimum distance of the dual code $C^\perp=\{x \in {\rm GF}(q)^k \mid \forall y \in C, (x, y)=0\}$, where $(,)$ is the standard inner product of ${\rm GF}(q)^k$. 

A \textit{BCH code} is a cyclic code in ${\rm GF}(q)^k$ with generator polynomial
\[
g(x) = {\rm lcm}\{M^{(a)}(x), M^{(a+1)}(x), \ldots, M^{(a+d-2)}(x)\},
\]
where $M^{(i)}(x)$ is the minimal polynomial of $\xi^i$ over ${\rm GF}(q)$, and $\xi$ is a primitive $k$-th root of unity in ${\rm GF}(q^m)$, an extension field of ${\rm GF}(q)$ \cite[Chapter 7, Section 6]{MSbook}. Here, we assume that $a + d - 2 \leq k - 1$. 
Since $\xi$ is a primitive $k$-th root of unity in ${\rm GF}(q^m)$, the order $q^m - 1$ of the multiplicative group must be divisible by $k$. 
The BCH code is a $[k, q^{k - \deg g}, d']_q$ code for some $d' \geq d$. The parameters satisfy the following conditions:
\[
0 \leq a \leq k - 1, \quad 2 \leq d \leq k - a + 1, \quad m \geq 1, \quad k \mid (q^m - 1).
\]

For a linear code $C$ in ${\rm GF}(q)^k$,  
let $C^\#$ denote the linear space in ${\rm GF}(q^m)^k$ spanned by the elements of $C$ over ${\rm GF}(q^m)$.  
On the other hand, for a linear code $C^\#$ in ${\rm GF}(q^m)^k$,  
let $C^\#|{\rm GF}(q)$ denote the linear space $C^\# \cap {\rm GF}(q)^k$ over ${\rm GF}(q)$. 
The following ${\rm GF}(q)$-linear function $T_m: {\rm GF}(q^m) \rightarrow {\rm GF}(q)$ is called the \textit{field trace}:  
\[
T_m(x) = x + x^q + x^{q^2} + \cdots + x^{q^{m-1}}.  
\]
For $x = (x_1, \ldots, x_k) \in {\rm GF}(q^m)^k$, we define  
\[
T_m(x) = (T_m(x_1), \ldots, T_m(x_k)) \in {\rm GF}(q)^k.  
\]
The dual code of $C^\#|{\rm GF}(q)$ can be expressed as  
\begin{equation} \label{eq:dualkey}
(C^\#|{\rm GF}(q))^\perp = T_m((C^\#)^\perp),
\end{equation}
for which see \cite[Theorem 11, Chapter 7]{MSbook}.

Let $C$ be the BCH code in ${\rm GF}(q)^k$, defined as above.  
A parity-check matrix of $C^\#$ is given by  
\begin{equation} \label{eq:parity_check}
    \begin{pmatrix}
    1 & \xi^a & \xi^{2a} & \cdots & \xi^{(k-1)a} \\
    1 & \xi^{a+1} & \xi^{2(a+1)} & \cdots & \xi^{(k-1)(a+1)}\\
    \vdots& \vdots & \vdots & & \vdots \\
    1 & \xi^{a+d-2} & \xi^{2(a+d-2)} & \cdots & \xi^{(k-1)(a+d-2)}
\end{pmatrix}
\end{equation}
since for any codeword of $C^\#$, the corresponding polynomial has $\xi^a, \xi^{a+1}, \ldots, \xi^{a+d-2}$ as its roots.
The rows of a parity-check matrix form generators of the dual code $(C^\#)^\perp$. This implies the following lemma. 

\begin{lemma}\label{lem:dualofC}
    Let $C$ be the BCH code in ${\rm GF}(q)^k$, defined as above. Then,  
\[
(C^\#)^\perp = \{(f(1), f(\xi), f(\xi^2), \ldots, f(\xi^{k-1})) \mid f \in {\rm GF}(q^m)_{[a,a+d-2]}[x] \},  
\]
where ${\rm GF}(q^m)_{[i,j]}[x]$ denotes the linear space of polynomials over ${\rm GF}(q^m)$ whose degrees lie in the range $[i,j]$.
\end{lemma}
\begin{proof}
The vectors corresponding to the monomials $f = x^a, x^{a+1}, \ldots, x^{a+d-2}$ form generators of the space.  
These generators coincide with the row vectors of the parity-check matrix~\eqref{eq:parity_check}.
\end{proof}

We present orthogonal arrays derived from the duals of extended BCH codes.  
The BCH code $C$ in ${\rm GF}(q)^k$ used in this study is defined by the generator polynomial  
\begin{equation} \label{eq:specific_g}
    g(x) = {\rm lcm}\{M^{(1)}(x), M^{(2)}(x), \ldots, M^{(t-1)}(x)\},
\end{equation}
with parameters $a = 0$, $d = t$, and $t \leq k=q^m-1$ ($\xi$ is the primitive element of ${\rm GF}(q^m)$). 
The {\it extended code} of $C^\# \subset {\rm GF}(q)^k$ is given by  
\[
\overline{C^\#} = \{(x_0, x_1, \ldots, x_k) \in {\rm GF}(q^m)^{k+1} \mid (x_1, \ldots, x_k) \in C^\#, \sum_{i=0}^{k} x_i = 0\}.
\]
From Lemma \ref{lem:dualofC}, one has
\[
(C^\#)^\perp = \{(f(1), f(\xi), f(\xi^2), \ldots, f(\xi^{k-1})) \mid f \in {\rm GF}(q^m)_{[1,t-1]}[x] \},   
\]
and hence $(\overline{C^\#})^\perp$ contains 
\[
F=\{(f(0),f(1), f(\xi), f(\xi^2), \ldots, f(\xi^{k-1})) \mid f \in {\rm GF}(q^m)_{[1,t-1]}[x] \}.   
\]
Note that $f(0)=0$ for each $f \in {\rm GF}(q^m)_{[1,t-1]}[x]$. 
The dimension of $(\overline{C^\#})^\perp$ is one greater than 
that of $(C^\#)^\perp$. 
Therefore, to obtain the whole code $(\overline{C^\#})^\perp$, the all-one vector must be added to the space $F$: 
\[
(\overline{C^\#})^\perp = \{(f(0),f(1), f(\xi), f(\xi^2), \ldots, f(\xi^{k-1})) \mid f \in {\rm GF}(q^m)_{t-1}[x] \}, 
\]
where ${\rm GF}(q^m)_{t-1}[x]$ denotes the linear space consisting of all polynomials of degree at most $t-1$. 
From \eqref{eq:dualkey}, we finally obtain the simple expression of the dual of the extended BCH code:
\begin{multline*}
    \overline{C}^\perp=(\overline{C^\#}|{\rm GF}(q))^\perp = T_m((\overline{C^\#})^\perp)\\
    =\{(T_m(f(0)),T_m(f(1)), T_m(f(\xi)), \ldots, T_m(f(\xi^{k-1}))) \mid f \in {\rm GF}(q^m)_{t-1}[x] \}. 
\end{multline*}

\begin{theorem} \label{thm:OA_2vr2}
Let $q$ be a prime power and $\xi$ a primitive element of ${\rm GF}(q^m)$ and $k=q^m-1$. 
    For any integer $t$ with $2\leq t\leq k$, the linear code 
    \[
    \mathcal{C}=\{(T_m(f(0)),T_m(f(1)), T_m(f(\xi)), \ldots, T_m(f(\xi^{k-1}))) \mid f \in {\rm GF}(q^m)_{t-1}[x] \}
    \]
    in ${\rm GF}(q)^{k+1}$ 
     forms an orthogonal array $\mathrm{OA}(N,q^m,q,t)$ with $N\leq q^{m(t-1)+1}= O((q^m)^{t-1})$ as $m\rightarrow \infty$.  
\end{theorem}
\begin{proof}
Since $\xi$ is a primitive element, $0,1,\xi,\ldots, \xi^{k-1}$ are mutually distinct. 
For each $t$-subset $\{i_1,\ldots,i_t\}$ of $\{0,1,\xi,\ldots, \xi^{k-1}\}$ and each $(j_1,\ldots,j_t) \in {\rm GF}(q^m)^t$,  
there exists $f\in {\rm GF}(q^m)_{t-1}[x]$ 
such that $f(i_s)=j_s$ for each $s \in \{1,\ldots, t\}$. 
In particular, we can take 
\[(j_1,\ldots,j_t)=(a,b,\ldots, b),(b,a,b,\ldots, b), \ldots, (b,\ldots,b,a)\]
for $a \in T_m^{-1}(1)$ and $b \in T_m^{-1}(0)$, and we can obtain $f_1,\ldots,f_t \in {\rm GF}(q^m)_{t-1}[x]$ such that 
\[
(T_m(f_s(i_1)),\ldots, T_m(f_s(i_t)))=e_s \qquad (s=1,2,\ldots, t), 
\]
where $\{e_s\}$ is the standard basis of ${\rm GF}(q)^t$. This implies that
any $t$-column vectors of the array corresponding to the code $\mathcal{C}$ are linearly independent. 
From Theorem \ref{thm:linear_code_and_OA}, 
the array forms an orthogonal array $\mathrm{OA}(N,k,q,t)$. 

The constant term of polynomial $f$ contributes the dimension of $\mathcal{C}$ only by one, and hence $\dim \mathcal{C} \le m(t-1)+1$. 
This implies that $N\leq q^{m(t-1)+1}$. 
\end{proof}

\begin{remark}
Reconfirming from the theory of BCH codes that the dimension of $ \mathcal{C} $ is at most $ m(t-1)+1 $, we obtain
\[
\dim \mathcal{C}=\dim {\overline{C}^\perp}=k+1-\dim \overline{C}=\deg g +1 \leq m(t-1)+1.
\]
The upper bound on $ \deg g $ follows from the definition \eqref{eq:specific_g}, and its exact value can be easily determined.  
If there are elements among $ \xi, \xi^2, \ldots, \xi^{t-1} $ that are mutually conjugate over $ GF(q) $ or belong to $ GF(q) $, then $ \deg g $ becomes smaller.    
\end{remark}
By puncturing a linear code at a coordinate position, we can obtain an orthogonal array with the same strength.

\begin{theorem} \label{thm:OA_any_d}
    Let $q$ be a prime power and $t$ an integer at least 2. 
    Then, for any integer $d$ with $t \leq d$, there exists an orthogonal array $\mathrm{OA}(N,d,q,t)$ with 
    $N<q^{t}d^{t-1}=O(d^{t-1})$ as $d\rightarrow \infty$. 
\end{theorem}

\begin{proof}
Let $m$ be the integer satisfying $q^{m-1}< d \leq q^m$. 
From Theorem~\ref{thm:OA_2vr2}, there exists a linear code $\mathcal{C}$ corresponding to an orthogonal array $\mathrm{OA}(N, q^m, q, t)$ with $N \leq q^{m(t-1)+1}$. 
Consider any projection $P:\mathcal{C} \rightarrow \mathrm{GF}(q)^d$ defined by selecting $d$ coordinates, i.e., 
$(x_1,\ldots,x_{q^m}) \mapsto (x_{i_1},\ldots,x_{i_d})$. 
Then the image $P(\mathcal{C})$ forms an orthogonal array $\mathrm{OA}(N', d, q, t)$, 
since the linear independence of any $t$ columns is preserved. 
Here, the run size $N'$ satisfies $
N' \leq N \leq q^{m(t-1)+1} < q^t d^{t-1}$. 
\end{proof}

\subsection{Explicit construction of equi-weighted $5$-designs with $O(d^4)$ points} \label{sect:appli0}

In this subsection, as an application of Corollary~\ref{cor:main_finite}, we provide an explicit construction of Gaussian $5$-designs in $\R^d$
with $O(d^4)$ points for any integer $d\geq 5$.

A {\it Gaussian $t$-design} in $\R^d$ is a weighted $t$-design with respect to the Gaussian measure $\mu(d\omega) = \pi^{-d/2} e^{- \sum_{i=1}^d \omega_i^2} d\omega_1 \cdots d\omega_d$ for $\omega=(\omega_1,\ldots,\omega_d)\in \R^d$~\cite{BB2005}.  
This measure on $\R^d$ is a Cartesian product of one-dimensional Gaussian measures:  
\[
\frac{1}{\pi^{d/2}}\exp\left(- \sum_{i=1}^d \omega_i^2\right) d\omega_1 \cdots d\omega_d = \prod_{i=1}^d\frac{1}{\sqrt{\pi}}\exp ( -\omega_i^2) d\omega_i. 
\]
We explicitly construct equi-weighted Gaussian $5$-designs in $\R^1$, and obtain higher-dimensional designs by applying the product rule.  
The size of the designs can be reduced using Corollary~\ref{cor:main_finite}, together with orthogonal arrays constructed in Subsection~\ref{sect:makeOA}.
 
The $k$-th moments of the one-dimensional Gaussian measure
\[
	a_k = \frac{1}{\sqrt{\pi}} \int_{-\infty}^\infty
	\omega^k e^{-\omega^2} d\omega,
	\quad k=0,1,\dotsc
\]
are calculated as 
\begin{equation}
\label{eq:moment_1}
	a_{2k} = \frac{(2k)!}{2^{2k} k!} = \frac{(2k-1)!!}{2^k},\quad
	a_{2k+1} = 0, \quad k=0,1,\ldots
\end{equation}

We consider a one-dimensional Gaussian $5$-design of type
\begin{equation} \label{eq:Gaussian0}
\frac{1}{2M+1} f(0) + \frac{1}{2M+1} \sum_{i=1}^M \{ f(z_i) + f(-z_i)
\} = \frac{1}{\sqrt{\pi}} \int_{-\infty}^\infty f(\omega) \; e^{-\omega^2}d\omega, \quad f \in \mathbb{R}_5[x], 
\end{equation}
whose weight is constant. 
By (\ref{eq:moment_1}), a design of type (\ref{eq:Gaussian0}) is equivalent to a solution of nonlinear equations of type
\begin{equation}
\label{eq:HK_0}
\begin{gathered} 
z_1^2 + \cdots + z_M^2 = \frac{2M+1}{4}, \\
z_1^4 + \cdots + z_M^4 = \frac{6M+3}{8},
\end{gathered}
\end{equation}
which we call {\em Hilbert-Kamke equations} following Cui, Xia and Xiang~\cite{CXX2019}.

Let $Z_i = z_i^2,\; i=1,\ldots,M$. 
We denote by $\mathcal{T}_{M-1}$ the $(M-1)$-dimensional standard simplex, and by $\mathbb{S}_r^{M-1}$ the $(M-1)$-dimensional sphere of radius $r := \sqrt{(6M+3)/8}$. Namely,
\[
\begin{gathered} 
\mathcal{T}_{M-1} = \{
(Z_1,\ldots,Z_M)
\in \mathbb{R}_{\ge 0}^M \mid
Z_1 + \cdots + Z_M = \tfrac{2M+1}{4}
\}, \\
\mathbb{S}_r^{M-1} =
\{
(Z_1,\ldots,Z_M) \in \mathbb{R}^M \mid Z_1^2 + \cdots + Z_M^2 = r^2
\}.
\end{gathered}
\]
It is not entirely obvious but shown that $\mathcal{C} := \mathcal{T}_{M-1} \cap \mathbb{S}_r^{M-1} \ne \emptyset$ if and only if $M \ge 3$. 

Let $\mathcal{T}_{M-1}'$ be the hyperplane containing $\mathcal{T}_{M-1}$.
Let $Q$ be the center of the circle $\mathcal{T}_{M-1}' \cap \mathbb{S}_r^{M-1}$. Then
\begin{equation}
\label{eq:HK_1}
Q = \Big( \frac{2M+1}{4M}, \ldots, \frac{2M+1}{4M} \Big).
\end{equation}
For any $P \in \mathcal{C}$, we have $OQ \perp PQ$ and so
\[
PQ^2 = OP^2 - OQ^2 = \frac{8M^2+2M-1}{16M} = \frac{(2M+1)(4M-1)}{16M}.
\]
Let
\begin{equation}
\label{eq:HK_2}
v = \Big(-4,-5,\ldots,-(M+2), \frac{(M-1)(M+6)}{2} \Big).
\end{equation}
Then we have $v \perp OQ$ and 
\[
\| v \|^2 = \frac{1}{12} M(M-1)(3M^2+37M+106).
\]
We take a point
$P=(Z_1,\ldots,Z_M)$ 
satisfying
\begin{equation}
\label{eq:HK_3}
\vec{QP} = \sqrt{\frac{(2M+1)(4M-1)}{16M}} \frac{1}{\| v \|} v.
\end{equation}
By (\ref{eq:HK_1}) and (\ref{eq:HK_2}),
we have
$Z_M > Z_1 > Z_2 > \cdots > Z_{M-1}$ 
and
\begin{align*}
Z_{M-1}
&= \frac{2M+1}{4M} - (M+2) \sqrt{\frac{12(2M+1)(4M-1)}{16M^2(M-1)(3M^2+37M+106)}} \\
& = \frac{\sqrt{2M+1}}{4M} \Big( \sqrt{2M+1} - (M+2) \sqrt{\frac{12(4M-1)}{(M-1)(3M^2+37M+106)}} \Big).
\end{align*}
Taking the difference between the squares of the first and second terms, we have
\[
2M+1 - \frac{12(M+2)^2(4M-1)}{(M-1)(3M^2+37M+106)}
= \frac{6M^4+23M^3-8M^2-287M-58}{(M-1)(3M^2+37M+106)}.
\]
Let $f(M) = 6M^4+23M^3-8M^2-287M-58$. Substituting $M = M'+3$, we have
\[
f(M) = 6(M')^4 + 95(M')^3 + 523(M')^2 + 934M'+116 > 0
\]
since $M' = M-3 \ge 0$. Therefore $X_{M-1}>0$ as desired.

We have explicitly constructed a Gaussian $5$-design of type (\ref{eq:Gaussian0}) and so by Proposition~\ref{prop:product_1} developed an explicit construction of equi-weighted Gaussian $5$-designs with $(2M+1)^d$ points in $d$ dimensions. To beat the curse of dimension, one may employ Corollary~\ref{cor:main_finite} and Theorem~\ref{thm:OA_any_d} when $q=2M+1$ is a prime power at least $7$. 

\begin{theorem}
\label{thm:HK_1}
For each $d \geq 5$, there exists a $d$-dimensional equi-weighted Gaussian $5$-design with $O(d^4)$ points.
\end{theorem}

\begin{proof}
Let $q$ be an odd prime power with $q \geq 7$, and set $M = (q - 1)/2$.  
Since $M \geq 3$, we can explicitly construct a one-dimensional equi-weighted Gaussian $5$-design $X$ with $q$ points, as stated above.  
By the product rule in Proposition~\ref{prop:product_1},  
$X^d$ is a $d$-dimensional equi-weighted Gaussian $5$-design with $q^d$ points.  

By Theorem~\ref{thm:OA_any_d}, there exists an orthogonal array $\mathrm{OA}(N, d, q, 5)$ with $N = O(d^4)$.  
As explained in Subsection~\ref{sec:OA_any_level},  
the entire set $X^d$ can be replaced by the subset corresponding to $\mathrm{OA}(N, d, q, 5)$ while preserving the equi-weighted $5$-design property.  
This completes the proof.
\end{proof}

Similar arguments will also work for product measures in general, though the details are omitted here. 

A special product measure is the {\it equilibrium measure} \begin{equation} \label{eq:equilibrium1}
d\phi = \frac{d\omega_1 \cdots d\omega_d}{\pi^d \prod_{j=1}^d \sqrt{1 - \omega_i^2}}, \qquad (\omega_1,\ldots,\omega_d) \in (-1,1)^d,
\end{equation}
for which an equi-weighted $t$-design is explicitly constructed for values of $t \ge 5$ in general (Theorem~\ref{thm:equilibrium1}). 
The measure $d\phi$, also written by the $d$th product of the Chebyshev measure
\begin{equation} \label{eq:Chebyshev1}
\frac{d\omega}{\pi\sqrt{1-\omega^2}}, \qquad \omega \in (-1,1),
\end{equation}
has recently received attention in optimal design of experiments~\cite{HL2024}.
A classical fact in numerical analysis and related areas is the Chebyshev-Gauss quadrature, namely
\begin{equation} \label{eq:Chebyshev2}
\int_{-1}^1 \frac{f(\omega)}{\pi \sqrt{1-\omega^2}} d\omega = \frac{1}{n} \sum_{i=1}^n f \Big( \cos\Big(\frac{2i-1}{2n}\pi \Big) \Big), \qquad f \in \mathbb{R}_{2n-1}[\omega].
\end{equation}
Proposition~\ref{prop:product_1} is then applicable to the construction of equi-weighted $(2n-1)$-designs of size $n^d$ for equilibrium measure $d\phi$.
In particular, when $n$ is a prime power, Corollary~\ref{cor:main_finite} with Theorem~\ref{thm:OA_any_d} produces small-sized weighted designs for $d\phi$.

\begin{theorem}
\label{thm:equilibrium1}
Let $n$ be a prime power and $t = 2n - 1$.
Then for any $d\geq t$, there exists an equi-weighted $t$-design with $O(d^{t-1})$ points for equilibrium measure $\pi^{-d} \prod_{i=1}^d (1-\omega_i^2)^{-1/2} d\omega_1 \cdots d\omega_d$ on $(-1,1)^d$.
\end{theorem}
\begin{proof}
   The proof proceeds similarly to that of Theorem~\ref{thm:HK_1}. 
\end{proof}

\subsection{Existence of weighted $t$-designs with $O(d^{t-1})$ points and isometric embeddings} \label{sect:appli1}

As already seen in Theorem~\ref{thm:equilibrium1}, we can obtain an explicit construction of equi-weighted $t$-designs with $O(d^{t-1})$ points for equilibrium measure on hypercube $(-1,1)^d$.
In this subsection, we prove an existence theorem of weighted $t$-designs with $O(d^{t-1})$ points for the Gaussian measure $\pi^{-d/2}e^{-\sum_{i=1}^d} d\omega_1 \cdots d\omega_d$ on $\mathbb{R}^d$.
The reader will realize that similar arguments work for other product measures.
What we emphasize here is the connection between Gaussian designs and isometric embeddings of the classical finite-dimensional Banach spaces.

Given distinct
$z_1,\ldots,z_n \in \mathbb{R}$, 
there uniquely exist $\lambda_1,\ldots,\lambda_n \in \mathbb{R}$ such that
\begin{equation}
\label{eq:admissible_1}
a_k =
\sum_{i=1}^n \lambda_i z_i^k, 
\quad k=0,1,\ldots,n-1,
\end{equation} 
where $a_k$ denote, as in (\ref{eq:moment_1}), the $k$th moments with respect to the Gaussian measure
$e^{-\omega^2}d \omega/\sqrt{\pi}$. 
The weight vector
$\lambda(z) = (\lambda_1,\ldots,\lambda_n)$ 
is {\em admissible} if $\lambda_i > 0$ for all $i$~\cite[p.173]{Kuijlaars1995}. The famous Gauss-Hermite quadrature ensures that the set of admissible weight vectors is nonempty.

\begin{lemma}
[cf.~Proposition 2.2 of \cite{Kuijlaars1995}]
\label{lem:OA_1}
The set of admissible weight vectors is an open subset of the hyperplane $\sum_{i=1}^n \lambda_i = 1$.
\end{lemma}
\begin{proof}
Suppose
$z_1,\ldots,z_n$ 
are mutually distinct reals.
It follows from partial differentiation of both sides of (\ref{eq:admissible_1}) that
\[
0 =
\sum_{i=1}^n z_i^k \frac{\partial \lambda_i}{\partial z_j} + \lambda_j \cdot k z_j^{k-1}, 
\]
or equivalently the Jacobian $(\partial \lambda_i/\partial x_j)$ is written by $M_1^{-1}M_2M_1M_3$, where
\[
\begin{gathered}
M_1 = 
\left(\begin{array}{cccc}
1 & 1 & \cdots & 1 \\
z_1 & z_2 & \cdots & z_n \\
\vdots & \vdots & & \vdots \\
z_1^{n-1} & z_2^{n-1} & \cdots & z_n^{n-1} \\ 
\end{array} \right),
\;
M_2 =
\left(\begin{array}{ccccc}
        0 &        0 & \cdots &  \cdots  & 0 \\
        1 &        0 & \cdots &  \cdots  & 0 \\
        0 &        2 & \ddots &              & 0 \\
\vdots & \vdots & \ddots &  \ddots  & \vdots \\
        0 &        0 & \cdots &      n-1  & 0 \\
\end{array} \right), \\
M_3 = {\rm diag}(\lambda_1,\lambda_2,\ldots,\lambda_n).
\end{gathered}
\]
Then
${\rm rank}(\partial \lambda_i/\partial z_j) = n-1$ 
and thus the map
$z \mapsto \lambda(z)$ 
is locally surjective onto the hyperplane $\sum_{i=1}^n \lambda_i = 1$.
\end{proof}

Now, let $\Omega = [q]^d$ where $[q] = \{1,2,\ldots,q\}$, and $\rho$ be the Hamming distance. We define $\tau:\Omega \rightarrow \mathbb{R}^{dq}$ to be
\[
\tau(i_1,\ldots,i_d) = (e_{i_1},\ldots,e_{i_d}),
\]
where $e_1,\dots,e_d$ are the standard basis of $\mathbb{R}^d$.
Since $d - \rho(x,y) = (\tau(x), \tau(y))$ for every pair $\{x,y\} \subset \Omega$, the map 
$\tau$ is affinely compatible with $\rho$. Clearly,
\[
\tau(\Omega) \subset \Big\{ (\omega_1,\ldots,\omega_q) \in \mathbb{R}^q \mid \sum_{i=1}^q \omega_i = 1 \Big\}^d \simeq \mathbb{R}^{d(q-1)}
\]
and the Hamming scheme $H(d,q)$ can be identified with the spherical polynomial space $(\Omega,\rho)$ of dimension $d(q-1)$.
For $\omega = (\omega_1,\ldots,\omega_q) \in \mathbb{R}^q$, where the entries $\omega_i$ are not necessarily distinct, we consider a linear map $\xi_\omega: \mathbb{R}^{dq} \rightarrow \mathbb{R}^d$ defined by
\[
\xi_\omega(y) =
y \left(\begin{array}{cccc}
\omega^\top &       0        & \cdots &        0       \\
      0        & \omega^\top & \cdots &        0       \\
  \vdots    &    \vdots   &            &    \vdots   \\
      0        &       0        & \cdots & \omega^\top \\
\end{array} \right).
\]
Then
\[
\tilde{\Omega} = (\xi_\omega \circ \tau)(\Omega) = \{ \omega_1,\ldots,\omega_q \}^d
\]
in the sense of multiset.
Let $\tilde{X} = (\xi_\omega \circ \tau)(X)$ as multiset, where $X$ is a $t$-design in $\Omega$, i.e. the set of runs of $\mathrm{OA}(|X|,d,q,t)$.
By Corollary~\ref{cor:Nozaki_1} with $\mu(x) \equiv 1/|\Omega|=1/|\tilde{\Omega}|$ and $w(x)\equiv 1/|X|=1/|\tilde{X}|$,
\begin{equation}
\label{eq:design_3}
\frac{1}{|\tilde{\Omega}|} \sum_{x \in \tilde{\Omega}} f(x) = \frac{1}{|\tilde{X}|} \sum_{x \in \tilde{X}} f(x) \ \text{ for every $f \in \mathbb{R}_t[x_1,\ldots,x_d]$}.
\end{equation}

Now by Lemma~\ref{lem:OA_1}, there exists a sufficiently large positive integer $M$ such that for any prime power $q \ge M$, there exist positive integers $q_1, \ldots, q_{t+1}$ with $\sum_i q_i = q$, and
distinct reals $z_1,\ldots,z_{t+1}$, 
for which
\begin{equation}
\label{eq:design_4}
a_k = \sum_{i=1}^{t+1} \frac{q_i}{q}
z_i^k, 
\quad k=0,1,\ldots,t.
\end{equation}

\begin{proposition}
\label{prop:OA_2}
With the above setup, we moreover assume that there exists an $\mathrm{OA}(N,d,q,t)$.
Then, there exists an $N$-points Gaussian $t$-design in $\R^d$. 
\end{proposition}
\begin{proof}
Regard  (\ref{eq:design_4}) as a $t$-design with $q$ points and take the $d$-fold product rule in Proposition \ref{prop:product_1}. The result then follows from (\ref{eq:design_3}).
\end{proof}

Combining Proposition~\ref{prop:OA_2} and Theorem~\ref{thm:OA_any_d}, we obtain the following result.

\begin{theorem}
\label{thm:OA_1}
Let $t$ be an integer at least 2. Then for fixed sufficiently large prime power $q$ and any integer $d \geq t$, 
there exists a Gaussian $t$-design in $\mathbb{R}^{d}$ with $N$ points, where $N < q^{t}d^{t-1} = O(d^{t-1})$ as $d \to \infty$. 
\end{theorem}

A Gaussian $2r$-design in $\mathbb{R}^d$ can be reduced to a weighted spherical design of index $2r$ (cf.~\cite{NS13}), which is equivalent to a Hilbert identity and a linear isometry between the classical finite-dimensional Banach spaces. 
Let $\ell_p^m$ denote the $m$-dimensional Banach space equipped with the $\ell_p$-norm.

\begin{proposition}
[\cite{LV1993}]
\label{prop:LV1993}
For $i=1, \ldots, N$, let ${y}_i = (y_{i1}, \ldots, y_{id}) \in \mathbb{S}^{d-1}$ with positive weight $w_i$. For any positive integer $r$, let
\begin{equation}
\label{eq:hilbert_1}
c_{d,r} = \frac{1}{|\mathbb{S}^{d-1}|} \int_{\mathbb{S}^{d-1}} \omega_{1}^{2r} \;
\nu(d\omega)
\end{equation}
where $\nu$ is the uniform measure on $\mathbb{S}^{d-1}$. 
We denote by ${\rm Hom}_{2r}(\mathbb{S}^{d-1})$ the space of homogeneous polynomials of degree $2r$ on $\mathbb{S}^{d-1}$. 
Then the following are equivalent:
\begin{enumerate}
\item[{\rm (i)}] (Spherical design of index $2r$). 
\begin{equation}
\label{eq:hilbert_2}
\frac{1}{|\mathbb{S}^{d-1}|} \int_{\mathbb{S}^{d-1}} f(\omega) \;
\nu(d\omega) 
= \sum_{i=1}^N w_i f(y_i) \ \text{ for every $f \in {\rm Hom}_{2r}(\mathbb{S}^{d-1})$}.
\end{equation}
\item[{\rm (ii)}] (Hilbert identity for $2r$-th powers). 
\begin{equation}
\label{eq:hilbert_3}
c_{d,r} (X_1^2 + \cdots + X_d^2)^r = \sum_{i=1}^N w_i (y_{i1} X_1 + \cdots + y_{id} X_d)^{2r}. 
\end{equation}
\item[{\rm (iii)}] (Isometric embeddings $\ell_2^d \hookrightarrow \ell_{2r}^N$). 
The following map $\iota: \ell_2^d \hookrightarrow \ell_{2r}^N$ is an isometric embedding: 
\begin{equation}
\label{eq:hilbert_4}
\iota(\omega) = \Big( \Big( \frac{w_1}{c_{d,r}} \Big)^{1/(2r)} \langle \omega, y_1 \rangle, \ldots, \Big( \frac{w_N}{c_{d,r}} \Big)^{1/(2r)} \langle \omega, y_N \rangle \Big),\; \omega \in \mathbb{R}^d.
\end{equation}
\end{enumerate}
\end{proposition}

Milman~\cite{Milman1988} (see also~\cite{LV1993})  proved that for any positive integers $d$ and $r$, there exists $N \le \dim {\rm Hom}_{2r}(\mathbb{R}^d) = O(d^{2r})$ for which 
there exists an isometric embedding $\ell_2^d \hookrightarrow \ell_{2r}^N$. The following is an asymptotic improvement of this result.

\begin{corollary}
\label{cor:OA_1}
Let $r$ be a positive integer. Then for fixed sufficiently large prime power $q$ and any integer $d \geq 2r$, 
there exists an integer $N \leq q^{2r}d^{2r-1} = O(d^{2r-1})$ as $d \to \infty$, such that there is an isometric embedding $\ell_2^{d} \hookrightarrow \ell_{2r}^N$.
\end{corollary}

\begin{proof}
The result follows from Theorem~\ref{thm:OA_1} and Proposition~\ref{prop:LV1993}.
\end{proof}

\section{Concluding remarks} \label{sec:concluding}

We established a general framework for reducing the sizes of weighted designs using designs in Euclidean polynomial spaces.
The origin of this type of design theory lies in the concept of $Q$-polynomial association schemes \cite{D73}, which are developed using univariate orthogonal polynomials.
The notion of $Q$-polynomial association schemes has been extended to the multivariate setting in \cite{BKZZ25}.
A typical example is the nonbinary Johnson association scheme introduced in \cite{TAG85}, which generalizes both the Johnson and Hamming schemes. 
The structure of its bivariate $Q$-polynomial association scheme has been made explicit in \cite{BKZZ24} and \cite{CVZZ24}.
One direction for future research is to develop a design theory for multivariate $Q$-polynomial schemes and to apply such designs to our reduction method.
It is clear that the first target should be the nonbinary Johnson scheme. 

The known explicit constructions of spherical designs \cite{BNOZ2022,X22} build higher-dimensional designs by stacking lower-dimensional ones, as in the case of the product rule.
Although we have not yet established size reduction methods for such designs, we anticipate potential progress in this direction.

In fact, Bannai et al. \cite{BNOZ2022} presented an explicit construction of unitary designs, which is similar to the idea of the product rule. We are also interested in developing a unitary analogue of our reduction method, partly motivated by its relevance to quantum information theory. 


\bigskip

\noindent
\textbf{Acknowledgments.} 
H. Nozaki was partially supported by JSPS KAKENHI Grant Numbers 22K03402 and 24K06688. 
M. Sawa was partially supported by JSPS KAKENHI Grant Number 24K06871
and the Early Support Program for Grand-in-Aid for Scientific Research A/B of Kobe University.


\end{document}